\documentclass[11pt]{amsart}
\usepackage{amsmath,amssymb}
\usepackage{amsthm,thmtools}
\usepackage{mathtools}
\usepackage{graphicx} 
\usepackage{complexity}
\usepackage[shortlabels]{enumitem}
\usepackage[margin=1in]{geometry}
\usepackage{tikz}
\usetikzlibrary{positioning}
\usepackage{ytableau}
\usepackage{float}
\usepackage{cleveref}

\newtheorem{theorem}{Theorem}[section]

\newtheorem{lemma}[theorem]{Lemma}

\newtheorem{proposition}[theorem]{Proposition}

\theoremstyle{definition}
\newtheorem{definition}[theorem]{Definition}

\newtheorem{remark}[theorem]{Remark}
\newtheorem{example}[theorem]{Example}

\def\P{{{\rm{\textsf{P}} }}}

\def\FP{{\rm{\textsf{FP}}}}

\def\SP{{{\rm{\textsf{\#P}}}}}

\DeclareMathOperator{\GL}{GL}

\title[A combinatorial interpretation for certain plethysm and Kronecker coefficients]{A combinatorial interpretation \\ for certain plethysm and Kronecker coefficients}
\author{Igor Pak}\address{Department of Mathematics, UCLA, Los Angeles, CA, 90095} \email{pak@math.ucla.edu} 
\author{Greta Panova}\address{Department of Mathematics, University of Southern California, Los Angeles, CA, 90089} \email{gpanova@usc.edu}
\author{Joshua P. Swanson}\address{Department of Mathematics, University of Southern California, Los Angeles, CA 90089} \email{swansonj@usc.edu}
\date{\today}

\begin{document}

\begin{abstract}We give explicit positive combinatorial interpretations for the plethysm coefficients $\langle s_\mu[s_\nu], s_\lambda\rangle$, when $\lambda$ has at most two rows, as counting certain marked trees.  The nontrivial cases reduce to $\nu=(k)$. In the special case $\mu=(n)$, this also yields a combinatorial interpretation for the corresponding rectangular Kronecker coefficient $g(\lambda, (n^k), (n^k))$. While it is easy to express these quantities as \emph{differences} of counting problems in the complexity class $\FP$, putting the problem in $\SP$, our interpretations give a positive counting formula over explicit marked trees.
\end{abstract}
\maketitle

\section{Introduction}

Two major open problems in algebraic combinatorics are to give combinatorial interpretations of the \emph{plethysm coefficients} \cite[Problem~9]{Stanley-positivity}  and the \emph{Kronecker coefficients} \cite[Problem~10]{Stanley-positivity}. The plethysm coefficient  gives the multiplicity of an irreducible Weyl module in the composition of two irreducible $\GL$-representations and can be formally defined as $a^\lambda_{\mu\nu}:= \langle s_\mu[s_\nu], s_\lambda\rangle$. The Kronecker coefficient  gives the multiplicity of an irreducible Specht module of $S_n$ in the tensor product of two other irreducible $S_n$-modules and can be defined as $g(\lambda,\mu,\nu):=\langle s_\lambda( \mathbf{x}\cdot \mathbf{y}), s_\mu(\mathbf{x})s_\nu(\mathbf{y})\rangle$. The definitions as multiplicities show they are nonnegative integers and pose the question of whether they count some ``nice'' discrete objects. 

Here we give new combinatorial interpretations for  plethysm coefficients $a^\lambda_{\mu,\nu}$ when $\lambda$ is a two-row partition and Kronecker coefficients $g(\lambda,(n^k),(n^k))$ when $\lambda$ has at most two rows.
These cases lie at the uncanny interface between problems which are easily seen to be in $\#\P$ (in fact, $\FP$, see the discussion in~\Cref{subsec:complexity}), yet the resulting combinatorial interpretation does not possess some desired aesthetic attributes.
Here we give a different combinatorial interpretation which arises from the highly nontrivial combinatorial proof of the unimodality of $q$-binomial coefficients of~\cite{KOH} and its extension~\cite{GOS}. The resulting combinatorial formulas lack some of the efficiency of numerical approaches, but they count explicit combinatorial objects in the most classical sense.

\begin{theorem}\label{thm:main_kron}
     The Kronecker coefficient $g(\lambda,(n^k),(n^k))$ for $\lambda = (nk-r,r)$ is equal to the number of \textcolor{blue}{marked KOH trees $\mathcal{T}(n,k,r)$}.
\end{theorem}

\begin{theorem}\label{thm:main_pleth}
    The plethysm coefficient $a^\lambda_{\mu\nu}$ for $\lambda=(kn-r,r)$ , $\nu=(k)$, and $\mu\vdash n$ is equal to the number of \textcolor{blue}{marked GOS trees $\mathcal{G}(\mu,k,r)$}.
\end{theorem}
This latter result covers all nontrivial cases of plethysm coefficients when $\lambda$ has at most two rows; see~\Cref{lem:pleth_all}.

The precise definitions of these marked trees are given in \Cref{sec:KOH-trees}, \Cref{sec:Kronecker}, and \Cref{sec:plethysm}. These trees have labels given by $(\alpha,a,b)$ where $a,b$ are integers and $\alpha \vdash b$ is a partition, and the relationships are all local. The marking refers to a tuple of integers associated to the leaves, and is the only non-local condition.


Our approach begins with the following well-known formulas. Let $p_r(n, k)$ denote the number of partitions with $r$ cells in the $k \times n$ rectangle, which can be computed as the coefficient at $q^r$ in the $q$-binomial coefficient:
$$p_r(n,k):=[q^r] \dbinom{n+k}{k}_q.$$
The Kronecker and plethysm coefficients for two-row partitions can be extracted as coefficients at $q^r$ as follows, see~\Cref{subsec:def_kron_pleth}
\begin{lemma}\label{lem:stab-diff}
    Suppose $\mu$ is an arbitrary partition, $\lambda = (N-r, r)$ has at most two rows, and $N = k|\mu|$ for some $k \geq 1$. Then
    \begin{equation}\label{eq:stab-diff-1}
      a^{\lambda}_{\mu,(k)}=\langle s_\mu[h_k], s_{(N-r, r)} \rangle = [q^r] (1-q) s_\mu(1, q, \ldots, q^k).
    \end{equation}
    When $\mu=(n)$, this specializes to
    \begin{equation}\label{eq:stab-diff-2}
        g((nk-r, r), (n^k), (n^k)) = p_r(n, k) - p_{r-1}(n, k)=[q^r] (1-q)\dbinom{n+k}{k}_q.
    \end{equation}
\end{lemma}

The positivity of the right-hand side of \eqref{eq:stab-diff-2} is a celebrated result originally due to Sylvester~\cite{sylvester}, who proved that the coefficients $\{p_r(n,k)\}_{r=0}^{nk}$ of each fixed $q$-binomial coefficient $\binom{n+k}{k}_q$ are a symmetric and unimodal sequence, i.e., 
\begin{equation}\label{eq:unimodality}
    p_0(n,k) \leq p_1(n,k) \leq \cdots \leq p_{ \lfloor nk/2\rfloor}(n,k) \geq \cdots \geq p_{nk}(n,k).
\end{equation}
Kathy O'Hara \cite{KOH} gave a long-sought combinatorial proof of Sylvester's unimodality result, which was subsequently reinterpreted algebraically by Zeilberger \cite{Zeilberger-KOH}, given a short algebraic proof by Macdonald \cite{Macdonald-KOH}, and extended to all $s_\mu(1, q, \ldots, q^k)$ by Goodman--O'Hara--Stanton \cite{GOS} using a key formula of Kirillov--Reshetikhin \cite{KR}. See the discussions in~\Cref{sec:background} for other related results, asymptotics, and complexity.

Our method for proving \Cref{thm:main_kron} and \Cref{thm:main_pleth} can be summarized as follows. Zeilberger's KOH formula for $\binom{n+k}{k}_q$ is unwound to give a sum of shifted products of $q$-integers, which are crucially all centered at $nk/2$. The terms are encoded by certain trees which we call \emph{KOH trees}. We then introduce a general technique (\Cref{lem:unimodal-product}) which takes as input combinatorial interpretations for the differences of successive coefficients of symmetric, unimodal polynomials and gives as output a combinatorial interpretation for the successive differences of their product. Applying this machinery to KOH trees yields the desired interpretation of \eqref{eq:stab-diff-2}; see \Cref{sec:Kronecker} and \Cref{thm:Kronecker}. More generally, applying it to the Goodman--O'Hara--Stanton formula yields a combinatorial interpretation of \eqref{eq:stab-diff-1}; see \Cref{sec:plethysm} and \Cref{thm:plethysm}.

The relationship between Kronecker coefficients and $q$-binomials in \eqref{eq:stab-diff-2} in \Cref{lem:stab-diff} was realized in~\cite{pak2014unimodality} to give another proof of the unimodality~\eqref{eq:unimodality} and extended via representation theoretic properties of the Kronecker coefficients to give strict unimodality in~\cite{pak2013strict} and better bounds in~\cite{pak2017bounds}. Strict unimodality was further derived through the KOH identity~\eqref{eq:KOH} in~\cite{zanello2015zeilberger,dhand2014combinatorial} and extended in~\cite{koutschan2023unified}. The tight asymptotics of $p_r(n,k)$ and $p_r(n,k)-p_{r-1}(n,k)$ were done via probabilistic methods in~\cite{melczer2020counting}. It is not hard to see that $a^{\lambda}_{(n),(k)}=p_r(n,k)-p_{r-1}(n,k)=g(\lambda,n^k,n^k)$ for $\lambda=(nk-r,r)$ being a two-row partition. The study of this difference via plethysms was more recently done in~\cite{orellana2024quasi}, giving different combinatorial interpretations for the difference in the cases when $k\leq 4$. A different approach towards such plethysms was presented in~\cite{gutierrez2024towards}. The generating functions of these plethysm coefficients are studied in~\cite{GOSSZ}. The relationship between two-row rectangular Kronecker and plethysm coefficients was investigated more deeply in~\cite{ikenmeyer2025field}. 



\section{Definitions and background}\label{sec:background}

We use standard notations for partitions and symmetric functions as in \cite{Macdonald,Stanley-EC2}. We denote by $\lambda \vdash n$ integer partitions of $n$, $\lambda=(\lambda_1,\ldots,\lambda_k)$ with $\lambda_1\geq \lambda_2 \geq \cdots \geq \lambda_k >0$, $\lambda_1+\cdots+\lambda_k=n$ and $\ell(\lambda)=k$ is their length. Let $p_r(n,k):=|\{ \lambda \vdash r, \lambda_1 \leq n, \ell(\lambda)\leq k\}|$ be the number of partitions whose Young diagram fits in a $k\times n$ rectangle. It is a classical fact that its generating function is given by the $q$-binomial coefficient:
$$\sum_{r=0}^{nk} p_r(n,k) q^r= \binom{n+k}{k}_q:= \prod_{i=1}^k \frac{1-q^{n+i}}{1-q^i}.$$

Additionally, we write
\begin{align*}
    m_j(\lambda) &\coloneqq \#\{i \mid \lambda_i = j\}, \\
    b(\lambda) &\coloneqq \sum_{i \geq 1} (i-1) \lambda_i.
\end{align*}

\subsection{Kronecker and plethysm coefficients via symmetric functions}\label{subsec:def_kron_pleth}
The irreducible representations of the symmetric group $S_n$ are given by the Specht modules $\mathbb{S}_\lambda$, and the multiplicities of their tensor product decompositions are the \emph{Kronecker coefficients} $g(\lambda,\mu,\nu)$:
$$\mathbb{S}_\lambda \otimes \mathbb{S}_\mu = \bigoplus_\nu \mathbb{S}_\nu^{\oplus g(\lambda,\mu,\nu)}.$$
The Kronecker coefficients are unchanged when permuting the three arguments.

The irreducible representations of $\GL_k(\mathbb{C})$ are the Weyl modules $V_\lambda$ indexed by partitions $\lambda$ with length $\ell(\lambda)\leq k$, and is given by a homomorphism $\rho_\lambda :\GL_k(\mathbb{C}) \to \GL_{r}(\mathbb{C})$, with $r=\dim V_\lambda$. The composition $\rho_\mu \circ \rho_\nu: \GL_k(\mathbb{C}) \to \GL_r(\mathbb{C})$ is a representation which decomposes into irreducible Weyl modules $V_\lambda$, each appearing with multiplicity $a^\lambda_{\mu,\nu}$ -- the \emph{plethysm coefficient}.

These multiplicites can be computed in practice through symmetric function identities and extraction of coefficients. Let $s_\lambda$ be the Schur function indexed by $\lambda$. Then

\begin{align}\label{eq:kron_schur}
    s_{\lambda}(\mathbf{x}\cdot\mathbf{y}) = \sum_{\mu,\nu} g(\lambda,\mu,\nu) s_\mu(\mathbf{x})s_\nu(\mathbf{y}),
\end{align}
where $\mathbf{x}=(x_1,x_2,\ldots)$, $\mathbf{y}=(y_1,y_2,\ldots)$ are two sets of variables and $\mathbf{x}\cdot\mathbf{y}=(x_1y_1,x_1y_2,\ldots,x_2y_1,\ldots)$ is the list of their pairwise products (in any order).

Similarly, we have 
\begin{align}\label{eq:pleth_schur}
    s_{\mu}[s_\nu(\mathbf{x})] = \sum_\lambda a^\lambda_{\mu,\nu} s_\lambda(\mathbf{x}),
\end{align}
where if $f(\mathbf{x}) = \mathbf{x}^{\alpha^1}+\mathbf{x}^{\alpha^2}+\cdots$ is the expansion of $f$ into monomials (appearing as many times as the multiplicity), then $g[f]:= g(\mathbf{x}^{\alpha^1},\mathbf{x}^{\alpha^2},\ldots)$ for any symmetric function $g(\mathbf{x})$.

While~\eqref{eq:stab-diff-1} and \eqref{eq:stab-diff-2} are easy to see, we give a proof for completeness.
\begin{proof}[Proof of~\Cref{lem:stab-diff}]
    Recall that $s_{(k)}(\mathbf{x}) = h_k(\mathbf{x})$ is the complete homogeneous symmetric function of degree $k$. To show \eqref{eq:stab-diff-1}, we have that 
    $$s_\mu[h_k(\mathbf{x})] = \sum_{\lambda} a^\lambda_{\mu,(k)} s_\lambda(\mathbf{x}),$$
    which holds for any substitution of the variables $\mathbf{x}$. The two-row Schur functions form a basis for the symmetric functions in two variables. We set $\mathbf{x}=(x_1,x_2,0,0,\ldots)$ and note that $s_\lambda(\mathbf{x})=0$ for $\ell(\lambda)\geq 3$. We have that $h_k(x_1,x_2)=x_1^k +x_1^{k-1}x_2+\cdots+x_2^k$, so
    $$s_\mu(x_1^k,x_1^{k-1}x_2,\ldots,x_2^k) = \sum_{r} a^\lambda_{\mu,(k)} s_\lambda(x_1,x_2),$$
    where $\lambda =(nk-r,r)$ and $\mu \vdash n$. Next we substitute $s_\lambda(x_1,x_2) = \frac{ x_1^{nk-r+1}x_2^r - x_2^{nk-r+1}x_1^r}{x_1-x_2}$ using Weyl's determinantal formula for the Schur functions and multiply both sides by $(x_1-x_2)$ to get
    $$(x_1-x_2)s_\mu(x_1^k,x_1^{k-1}x_2,\ldots,x_2^k) = \sum_r a^{\lambda}_{\mu,(k)} (x_1^{nk-r+1}x_2^r - x_2^{nk-r+1}x_1^r).$$
    Since these are homogenous polynomials of degree $nk+1$, we can dehomogenize by setting $x_1=1,x_2=q$ and derive
    $$(1-q)s_\mu(1,q,\ldots,q^k) = \sum_{r=0}^{\lfloor nk/2 \rfloor} a^{\lambda}_{\mu,(k)} (q^r - q^{nk-r+1}) = \sum_{r=0}^{\lfloor nk/2 \rfloor} a^{(nk-r,r)}_{\mu,(k)}q^r - \sum_{j=\lceil nk/2 \rceil +1}^{nk+1} a^{(j-1,nk-j+1)}_{\mu,(k)}q^j .$$
    Now we can extract $a^{(nk-r,r)}_{\mu,(k)}$ as the coefficient at $q^r$. 

    To show~\eqref{eq:stab-diff-2}, we set $\mathbf{y}=(1,q,0,\ldots)$ and $\nu=(nk-r,r)$, which gives, similarly to the above expansion,
    $$g(\lambda,\mu,\nu) = [q^r](1-q) \langle s_\lambda(x_1,x_2,\ldots,qx_1,qx_2,\ldots), s_\mu(x_1,x_2,\ldots)\rangle_{\mathbf{x}} $$
    with the Hall inner product over the symmetric function ring with variables $\mathbf{x}$. Using skew Schur functions, one may show $s_\lambda(\mathbf{x},q\mathbf{x}) = \sum_{\alpha,\beta} c^\lambda_{\alpha\beta}q^{|\beta|}s_{\alpha}(\mathbf{x})s_\beta(\mathbf{x})$, where $c^\lambda_{\alpha\beta}$ are the Littlewood--Richardson coefficients. Thus
    $$\langle s_\lambda(x_1,x_2,\ldots,qx_1,qx_2,\ldots), s_\mu(x_1,x_2,\ldots)\rangle_{\mathbf{x}} = \sum_{\alpha,\beta} q^{|\beta|}c^\lambda_{\alpha\beta} \langle s_{\alpha}(\mathbf{x})s_\beta(\mathbf{x}),s_\mu(\mathbf{x})\rangle=\sum_{\alpha,\beta} q^{|\beta|}c^\lambda_{\alpha\beta} c^\mu_{\alpha\beta}.$$
    Finally, we realize, say by the Littlewood--Richardson rule, that when $\lambda=(n^k)$ is a rectangle we have that $c^\lambda_{\alpha\beta}=1$ iff $\beta_i=n-\alpha_{k+1-i}$ for each $i$, and $0$ otherwise. That is, $\alpha$ and $\beta$ are complementary partitions inside the rectangle. Hence when $\lambda=\mu=(n^k)$, the above sum is just $\sum_{\beta \subset (n^k)} q^{|\beta|}= \sum_r p_r(n,k)q^r$ by definition, and the identity follows.
\end{proof}

\begin{lemma}\label{lem:pleth_all}
Let $\nu \vdash r$ and $\mu \vdash m$ and $\lambda \vdash mr$, such that $\ell(\lambda) \leq 2$. Then
\begin{equation}
    a^\lambda_{\mu,\nu}=
    \begin{cases} 
      0& \text{ if $\ell(\nu) \geq 3$} \\
      a^{\theta}_{\mu,(k)} & \text{ if $\ell(\nu) \leq 2$, $\nu_1-\nu_2=k$, $\theta = (\lambda_1-m\nu_2, \lambda_2-m\nu_2) \vdash mk$} \\
      0& \text{ if $\ell(\nu) \leq 2$ and $\lambda_2 < m\nu_2$}.
    \end{cases}
\end{equation}
\begin{proof}
    Using the same Schur function expansion as above, we restrict to $\mathbf{x}=(x_1,x_2,0,\ldots)$ and get
    $$s_\mu[s_\nu(x_1,x_2)] = \sum_{\lambda} a^\lambda_{\mu,\nu} s_\lambda(x_1,x_2).$$
    When $\ell(\lambda) \leq 2$, we have that $s_{\lambda}(x_1,x_2) \neq 0$. If $\ell(\nu)\geq 3$ then $s_\nu(x_1,x_2)=0$ and the left-hand side above is 0. Thus all coefficients at the nonzero $s_\lambda(x_1,x_2)$ should vanish and so $a^\lambda_{\mu,\nu}=0$, covering the first case. 

    Now let $\nu=(b+k,b)$ for some $b$. We have $s_\nu(x_1,x_2)=(x_1x_2)^b h_k(x_1,x_2)$, so
    
    $$(x_1x_2)^{bm} s_\mu[h_k(x_1,x_2)]= s_\mu[ (x_1x_2)^b h_k(x_1,x_2)] = \sum_{\lambda} a^{\lambda}_{\mu,\nu} s_\lambda(x_1,x_2). $$
    
    Writing $s_\lambda = (x_1x_2)^{\lambda_2}\; \frac{ x_1^{\lambda_1+1-\lambda_2} - x_2^{\lambda_1+1-\lambda_2}}{x_1-x_2}$ and multiplying both sides by $(x_1-x_2)$, we see that on the right-hand side only monomials divisible by $(x_1x_2)^{bm}$ should remain. Thus $a^\lambda_{\mu,\nu}=0$ when this is not true, i.e., $\lambda_2  <mb$. 

    Finally, let $\lambda_2 \geq mb$, so $\lambda=\theta +(mb,mb)$ and $s_\lambda(x_1,x_2)=(x_1x_2)^{bm} s_\theta(x_1,x_2)$. Canceling the monomials $(x_1x_2)^{bm}$ on both sides, we are left with
    $$s_\mu[h_k(x_1,x_2)] = \sum_{\lambda} a^{\lambda}_{\mu,\nu} s_\theta(x_1,x_2),$$
    and by expanding the left-hand side again via~\eqref{eq:pleth_schur} we can identify $a^\lambda_{\mu,\nu} = a^\theta_{\mu,(k)}$.
\end{proof}

\end{lemma}

\subsection{Computational complexity}\label{subsec:complexity}
Recent work \cite{IP-sharp-P,IPP-characters,Pak-comb-interp,Pan23} has proposed using the complexity class \#\P to formalize the notion of \emph{positive combinatorial interpretation}. 
Informally, this is the class of counting problems that enumerate objects, called witnesses, each of which  is verifiable in polynomial time in the input size. While a \emph{positive combinatorial interpretation} has never been formally defined, it is interpreted as counting ``nice objects'' and hence the $\SP$ formalism is the closest to it. In particular, if a counting problem is not in $\#\P$ under reasonable complexity theoretic assumptions, then it also does not have a nice positive combinatorial interpretation, as shown in~\cite{IPP-characters} for the squares of symmetric group characters, for example. One of the flagship examples of ``positive combinatorial interpretation'' is the problem of computing the Littlewood--Richardson coefficients $c_{\mu\nu}^\lambda = \langle s_\mu s_\nu, s_\lambda\rangle$. They may not have an explicit closed-form formula, but are equal to the number of certain tableaux, which are the witnesses in the $\#\P$ formalism.  

It is not hard to see that the numbers $p_r(n,k)$ can be computed in time $O(nk)$ by a dynamic programming approach using the recursion 
$$p_r(n,k) = p_r(n,k-1) + p_{r-k}(n-1,k).$$
Thus, the problem of computing $p_r(n,k)$ and also $p_r(n,k) - p_{r-1}(n,k)$ (which is $\geq 0$) is in the class $\FP$ of counting functions computable in polynomial time. Since we know that  $\FP \subset \#\P$, we also have that the problem of computing that Kronecker coefficient is in $\#\P$. A combinatorial interpretation could be given by: $g((nk-r,r),(n^k),(n^k))$ counts the numbers in the interval $[1,\ldots,p_r(n,k)-p_{r-1}(n,k)]$, whose upper bound is computed in polynomial time via the recursion. Similarly, one can use the $q$-hook-content formula to compute the whole polynomial expansion efficiently 
$$s_\mu(1,q,\ldots,q^k) = q^{b(\mu)} \prod_{(i,j)\in [\mu]} \frac{1-q^{k+1+j-i}}{1-q^{\mu_i+\mu'_j-i-j+1} },$$
and extract the coefficients at $q^r$ and $q^{r-1}$. That also gives that computing the plethysm coefficient $a^{(nk-r,r)}_{\mu,(k)}$ is in $\FP \subset \#\P$ with a similar combinatorial interpretation.  Such an interpretation, however, may seem unsatisfactory as a ``positive combinatorial interpretation'', as it requires us to compute the final value first and only afterwards identify the ``objects'' it counts.  It is also not sufficient to prove containment in lower counting complexity classes like $\#\mathsf{L}$ as first considered in~\cite{GIP}.


\section{Binomial identities and trees}
\subsection{The KOH identity and KOH trees}\label{sec:KOH-trees}

Building on work of O'Hara \cite{KOH}, Zeilberger \cite{Zeilberger-KOH} gave the following formula for the $q$-binomial coefficients. If $\lambda$ is a partition, let $\lambda'$ denote the conjugate partition. 

\begin{theorem}[The KOH identity]\label{thm:KOH}
    We have
    \begin{equation}\label{eq:KOH}
      \binom{n+k}{k}_q = \sum_{\lambda \vdash k} q^{2b(\lambda)} \prod_{j \geq 1} \binom{(n+2)j - 2(\lambda_1'+\cdots+\lambda_j')+m_j(\lambda)}{m_j(\lambda)}_q.
    \end{equation}
\end{theorem}

\noindent The key observation is that all of the summands in \eqref{eq:KOH} are symmetric about $nk/2$, from which Sylvester's unimodality result follows easily by induction. We reinterpret \eqref{eq:KOH} using certain trees.

\begin{definition}\label{def:KOH-tree}
    A \emph{KOH tree} (see \Cref{fig:KOH-tree} for an example) is a rooted tree with linearly ordered children where each vertex has a label of the form $(\mu, a, b)$ for some integers $a \geq 0$, $b \geq 1$ and some partition $\mu \vdash b$, subject to the following constraints.
    \begin{enumerate}[(i)]
        \item Each leaf node $v$ has label of the form $((1), a, 1)$, which we often abbreviate as $a$.
        \item Each non-leaf node $v$ has label of the form $(\mu, a, b)$ for $b \geq 2$. Moreover, if the distinct row lengths of $\mu$ are $j_1 < \cdots < j_\ell$, then $v$ has precisely $\ell$ children with some labels $(\mu^{(i)}, a^{(i)}, b^{(i)})$, where we require
        \begin{align*}
            a^{(i)} &= (a+2) j_i - 2(\mu_1'+\cdots+\mu_{j_i}'), \\
            b^{(i)} &= m_{j_i}(\mu).
        \end{align*}
    \end{enumerate}
    The \emph{type} of a KOH tree is the pair $(n, k)$ where the root is labeled by $(\lambda, n, k)$. Write $\mathcal{T}(n, k)$ for the collection of KOH trees of type $(n, k)$.
\end{definition}

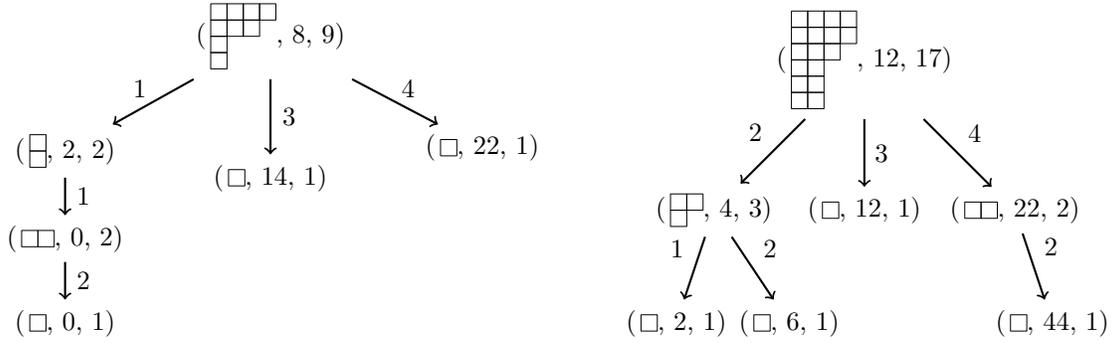
\begin{figure}[hbtp]
    \begin{center}
    \begin{tikzpicture}[font=\small]
        \node (root) {$(\,\vcenter{\hbox{\scalebox{0.4}{\ydiagram{4,3,1,1}}}},\,8,\,9)$};
        \node[below left=6mm and 8mm of root] (left) {$(\,\vcenter{\hbox{\scalebox{0.4}{\ydiagram{1,1}}}},\,2,\,2)$};
        \node[below=10mm of root] (middle) {$(\,\vcenter{\hbox{\scalebox{0.4}{\ydiagram{1}}}},\,14,\,1)$};
        \node[below right=6mm and 8mm of root] (right) {$(\,\vcenter{\hbox{\scalebox{0.4}{\ydiagram{1}}}},\,22,\,1)$};
        \node[below=5mm of left] (leftchild) {$(\,\vcenter{\hbox{\scalebox{0.4}{\ydiagram{2}}}},\,0,\,2)$};
        \node[below=5mm of leftchild] (leftchildchild) {$(\,\vcenter{\hbox{\scalebox{0.4}{\ydiagram{1}}}},\,0,\,1)$};
        
        \draw[thick, ->] (root) -- (left)        node[midway, xshift=-5pt, yshift=5pt] {$1$};
        \draw[thick, ->] (root) -- (middle)       node[midway, xshift=7pt] {$3$};
        \draw[thick, ->] (root) -- (right)       node[midway, xshift=5pt,  yshift=5pt] {$4$};
        \draw[thick, ->] (left) -- (leftchild)  node[midway, xshift=7pt] {$1$};
        \draw[thick, ->] (leftchild) -- (leftchildchild) node[midway, xshift=7pt] {$2$};
    \end{tikzpicture}
$\qquad$\begin{tikzpicture}[font=\small]
  \node (r) at (0,0) { $(\,\vcenter{\hbox{\scalebox{0.4}{\ydiagram{4,4,3,2,2,2}}}},\,12,\,17)$ };

  \node (n1) at (-2,-2) { $(\,\vcenter{\hbox{\scalebox{0.4}{\ydiagram{2,1}}}},\,4,\,3)$ };
  \node (n2) at (0,-2)  { $(\,\vcenter{\hbox{\scalebox{0.4}{\ydiagram{1}}}},\,12,\,1)$ };
  \node (n3) at (2,-2)  { $(\,\vcenter{\hbox{\scalebox{0.4}{\ydiagram{2}}}},\,22,\,2)$ };

  \draw[thick, ->] (r) -- (n1) node[midway, above left] {$2$};
  \draw[thick, ->] (r) -- (n2) node[midway, right]      {$3$};
  \draw[thick, ->] (r) -- (n3) node[midway, above right]{$4$};

  \node (a1) at (-2.5,-3.5) {$(\,\vcenter{\hbox{\scalebox{0.4}{\ydiagram{1}}}},\,2,\, 1)$};
  \node (a2) at (-1,-3.5) {$(\,\vcenter{\hbox{\scalebox{0.4}{\ydiagram{1}}}},\,6,\, 1)$};
  \draw[thick, ->] (n1) -- (a1) node[midway, above left] {$1$};
  \draw[thick, ->] (n1) -- (a2) node[midway, above right]{$2$};

  \node (b2) at (2.5,-3.5) { $(\,\vcenter{\hbox{\scalebox{0.4}{\ydiagram{1}}}},\,44,\,1)$ };
  \draw[thick, ->] (n3) -- (b2) node[midway, above right]{$2$};
\end{tikzpicture}
    \end{center}
    \caption{(\textsc{Left}) A KOH tree $T$ of type $(8, 9)$. Edges are labeled with the distinct row lengths in the parent's partition. The leaf multiset is $\mathcal{L}(T) = \{0, 14, 22\}$ and the corresponding term of \eqref{eq:KOH-trees} contributing to $\binom{8+9}{9}_q$ is $q^{(72-0-14-22)/2}[0+1]_q [14+1]_q [22+1]_q$. (\textsc{Right}) A KOH tree of type $(12,17)$ with $\mu=(4,4,3,2,2,2)$ and leaf multiset $\mathcal{L}(T)=\{2,6,12,44\}$.}
    \label{fig:KOH-tree}
\end{figure}

\begin{remark}
    KOH trees are finite. One way to see this is to observe that only $\lambda = (k)$ yields $b(\lambda)=0$ in \eqref{eq:KOH}, giving the term $[nk+1]_q$, with all other terms contributing only to strictly interior coefficients since $\binom{n+k}{k}_{q=0} = 1$. More precisely, the children $(\mu^{(i)}, a^{(i)}, b^{(i)})$ of $(\mu, a, b)$ in a valid KOH tree satisfy $a^{(i)}b^{(i)} \leq ab$, with equality holding if and only if $\mu = (b)$. The right-hand expression for $a^{(i)}$ in \Cref{def:KOH-tree} may be negative, so some $\lambda \vdash k$ do not contribute to \eqref{eq:KOH}. 
\end{remark}

Given a KOH tree $T \in \mathcal{T}(n, k)$, write $\mathcal{L}(T)$ for the multiset of labels of its leaves and set
  \[ \sigma(T) \coloneqq nk - \sum_{a \in \mathcal{L}(T)} a. \]
Unwinding the recursion in \Cref{thm:KOH} immediately yields the following.

\begin{proposition}\label{prop:KOH-trees}
    Suppose $n, k \geq 1$. Then
    \begin{equation}\label{eq:KOH-trees}
        \binom{n+k}{k}_q
          = \sum_{T \in \mathcal{T}(n, k)} q^{\sigma(T)/2}\prod_{a \in \mathcal{L}(T)} [a+1]_q.
    \end{equation}
\end{proposition}

\subsection{The GOS identity and GOS trees}

The left-hand side of \eqref{eq:KOH} has a well-known interpretation as a principal specialization of a complete homogeneous symmetric polynomial (see, e.g., \cite[Prop.~7.8.3]{Stanley-EC2}),
  \[ h_n(1, q, \ldots, q^k) = \binom{n+k}{k}_q. \]
Indeed, the principal specializations $s_\lambda(1, q, \ldots, q^k)$ are well-known to have unimodal coefficient sequences. This can be proved combinatorially with a generalization of the KOH identity given by Goodman--O'Hara--Stanton \cite{GOS}. As they observe, this more general identity largely follows from a formula for $q$-Kostka polynomials due to Kirillov--Reshetikhin \cite[Thm.~4.4]{KR}, which is quite similar to \eqref{eq:KOH}. The statement of the identity in \cite{GOS} is implicit and relies on notation in \cite{KR}, so for completeness we give a self-contained statement here.

If $\kappa$ is a partition, recall that $\kappa'$ denotes the conjugate partition. Set
  \[ Q_j(\kappa) \coloneqq \sum_{x=1}^j \kappa_x'. \]

\begin{definition}
   Suppose $\lambda$ is a partition where $|\lambda|=n$ and $\ell(\lambda)=\ell$. An \emph{admissible $\lambda$-configuration} is a sequence of partitions
      \[ \underline{\nu} = (\nu^{(0)}, \nu^{(1)}, \ldots, \nu^{(\ell)}) \]
    such that
    \begin{enumerate}[(i)]
        \item $\nu^{(0)} = (1^n)$,
        \item $|\nu^{(i)}| = \sum_{j \geq i+1} \lambda_j$ for $0 \leq i \leq \ell$ (so $\nu^{(\ell)} = \varnothing$), and
        \item $P_j^{(i)}(\underline{\nu}) \coloneqq Q_j(\nu^{(i+1)}) - 2Q_j(\nu^{(i)}) + Q_j(\nu^{(i-1)}) \geq 0$ for all $1 \leq i < \ell$ and $1 \leq j \leq n$.
    \end{enumerate}
    Further, write $\alpha^{(i)} \coloneqq (\nu^{(i)})'$ and set
    \begin{align*}
      m(\underline{\nu}) &\coloneqq \alpha^{(1)}_1 \\
      \tau(\underline{\nu}) &\coloneqq \sum_{\substack{1 \leq i < \ell \\ 1 \leq j \leq n}} \alpha_j^{(i)} (\alpha_j^{(i)} - \alpha_j^{(i+1)}).
    \end{align*}
\end{definition}

The following is essentially stated as \cite[(1.3)]{GOS} (though the $n$ on their right-hand side is not the same $n$ as on their left-hand side). It is in turn largely a reformulation of \cite[Thm.~4.2,~eq.~(4.3)]{KR}.

\begin{theorem}[The ``GOS identity'']\label{thm:GOS}
    If $\lambda$ is a partition with $|\lambda|=n$ and $\ell(\lambda)=\ell$, then
    \begin{equation}
        s_\lambda(1, q, \ldots, q^k)
          = \sum_{m=0}^k \binom{n+k-m}{k-m}_q
          \sum_{\substack{\underline{\nu} \\ \text{s.t. }m(\underline{\nu})=m}} q^{\tau(\underline{\nu})} \prod_{\substack{1 \leq i < \ell \\ 1 \leq j \leq n}} \binom{P_j^{(i)}(\underline{\nu}) + m_j(\nu^{(i)})}{m_j(\nu^{(i)})}_q,
    \end{equation}
    where the sum is over $\lambda$-admissible configurations $\underline{\nu}$. Each summand is symmetric about $nk/2$.
\end{theorem}

\begin{remark}
    When $\lambda = (n)$ (or $\lambda = (1^n)$), there is a single admissible configuration, $\underline{\nu} = ((1^n), \varnothing)$ (or $\underline{\nu} = ((1^n), (1^{n-1}), \ldots, (1), \varnothing)$) and the GOS formula reduces to a single $q$-binomial coefficient. In this sense, the KOH formula is genuinely different than the GOS formula and is not directly a special case of it.
\end{remark}

\begin{definition}\label{def:GOS-tree}
    A \emph{GOS tree} (see \Cref{fig:GOS-tree} for an example) is a rooted tree subject to the following constraints.
    \begin{enumerate}[(i)]
        \item The root has label $(\lambda, \underline{\nu}, k)$ for some partition $\lambda$, some $k \geq 0$, and some $\lambda$-admissible configuration $\underline{\nu}$ with $m(\underline{\nu}) \leq k$.
        \item For each $1 \leq i < \ell(\lambda)$ and $1 \leq j \leq |\lambda|$ for which $m_j(\nu^{(i)}) \neq 0$, there is a unique corresponding child of the root, which is a KOH tree of type $(P_j^{(i)}(\underline{\nu}), m_j(\nu^{(i)}))$. The edge from the root to this child is labeled by $(i, j)$.
        \item If $m(\underline{\nu}) < k$, there is exactly one additional child of the root, which is a KOH tree of type $(|\lambda|, k - m(\underline{\nu}))$. The edge from the root to this child is unlabeled.
    \end{enumerate}
    The \emph{type} of a GOS tree is the pair $(\lambda, k)$. Write $\mathcal{G}(\lambda, k)$ for the collection of GOS trees of type $(\lambda, k)$.
\end{definition}

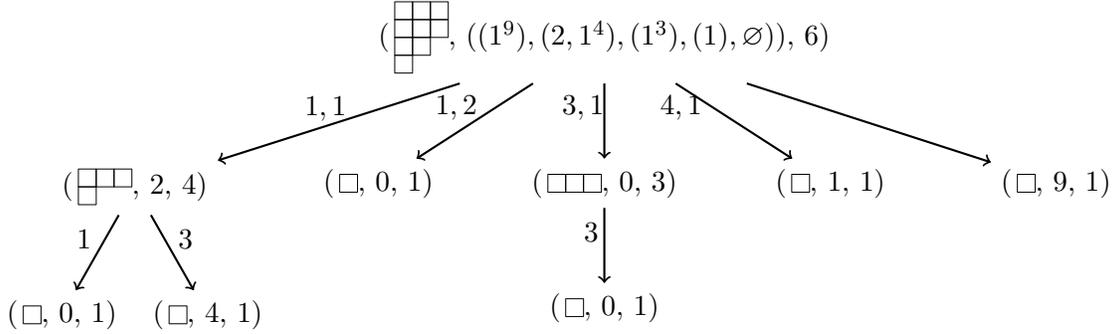
\begin{figure}[hbtp]
    \begin{center}
    \begin{tikzpicture}
        \node (root) {$(\,\vcenter{\hbox{\scalebox{0.4}{\ydiagram{3,3,2,1}}}},\,((1^9), (2, 1^4), (1^3), (1), \varnothing)),\,6)$};
        \node[below left=10mm and 20mm of root] (a) {$(\,\vcenter{\hbox{\scalebox{0.4}{\ydiagram{3,1}}}},\,2,\,4)$};
        \node[below left=10mm and -10mm of root] (b) {$(\,\vcenter{\hbox{\scalebox{0.4}{\ydiagram{1}}}},\,0,\,1)$};
        \node[below=10mm of root] (c) {$(\,\vcenter{\hbox{\scalebox{0.4}{\ydiagram{3}}}},\,0,\,3)$};
        \node[below right=10mm and -10mm of root] (d) {$(\,\vcenter{\hbox{\scalebox{0.4}{\ydiagram{1}}}},\,1,\,1)$};
        \node[below right=10mm and 20mm of root] (e) {$(\,\vcenter{\hbox{\scalebox{0.4}{\ydiagram{1}}}},\,9,\,1)$};

        \node[below left=10mm and -10mm of a] (aa) {$(\,\vcenter{\hbox{\scalebox{0.4}{\ydiagram{1}}}},\,0,\,1)$};
        \node[below right=10mm and -10mm of a] (ab) {$(\,\vcenter{\hbox{\scalebox{0.4}{\ydiagram{1}}}},\,4,\,1)$};
        \node[below=10mm of c] (ca) {$(\,\vcenter{\hbox{\scalebox{0.4}{\ydiagram{1}}}},\,0,\,1)$};
        
        \draw[thick, ->] (root) -- (a)        node[midway, xshift=-5pt, yshift=5pt] {$1,1$};
        \draw[thick, ->] (root) -- (b)        node[midway, xshift=-7pt, yshift=5pt] {$1,2$};
        \draw[thick, ->] (root) -- (c)        node[midway, xshift=-8pt, yshift=5pt] {$2,1$};
        \draw[thick, ->] (root) -- (d)        node[midway, xshift=-20pt, yshift=5pt] {$3,1$};
        \draw[thick, ->] (root) -- (e)        node[midway, xshift= 5pt, yshift=5pt] {};

        \draw[thick, ->] (a) -- (aa)          node[midway, xshift=-5pt, yshift=5pt] {$1$};
        \draw[thick, ->] (a) -- (ab)          node[midway, xshift= 5pt, yshift=5pt] {$3$};
        \draw[thick, ->] (c) -- (ca)          node[midway, xshift=-5pt, yshift=5pt] {$3$};
    \end{tikzpicture}
    \end{center}
    \caption{A GOS tree of type $((3, 3, 2, 1), 6)$.}
    \label{fig:GOS-tree}
\end{figure}

If $T \in \mathcal{G}(\lambda, k)$, set
  \[ \sigma(T) \coloneqq |\lambda|k - \sum_{a \in \mathcal{L}(T)} a. \]
Combining \Cref{thm:GOS} and \Cref{prop:KOH-trees} immediately yields the following.

\begin{proposition}\label{prop:GOS-trees}
    Suppose $\lambda$ is a partition and $k \geq 1$. Then
    \begin{equation}\label{eq:GOS-trees}
        s_\lambda(1, q, \ldots, q^k)
          = \sum_{T \in \mathcal{G}(\lambda, k)} q^{\sigma(T)/2} \prod_{a \in \mathcal{L}(T)} [a+1]_q.
    \end{equation}
\end{proposition}

\section{The Kronecker case}\label{sec:Kronecker}

\subsection{Products of symmetric, unimodal polynomials}
It is very well-known that the product of symmetric, unimodal polynomials with non-negative coefficients remains symmetric and unimodal (see, e.g., \cite[Prop.~1]{Stanley-LCC}). In fact, the standard proof can be strengthened to give the successive differences of the product in terms of those of the factors.

\begin{lemma}\label{lem:unimodal-product}
    Let $A(q) = \sum_{i=0}^r a_i q^i$ and $B(q) = \sum_{i=0}^s b_i q^i$ be symmetric, unimodal polynomials with non-negative coefficients where $a_r, b_s \neq 0$. Set $C(q) \coloneqq A(q) B(q) = \sum_{i=0}^{r+s} c_i q^i$. Then $C(q)$ is symmetric and unimodal with
    \begin{equation}\label{eq:unimodal-product}
        c_k - c_{k-1} = \sum_{(i, j) \in R_k(r, s)} (a_i - a_{i-1})(b_j - b_{j-1})
    \end{equation}
    for $0 \leq k \leq (r+s)/2$, where
    \[
        R_k(r, s) \coloneqq \left\{(i, j) \in \mathbb{Z}^2 \;\middle|\;
        \begin{array}{l}
            0 \leq i \leq \frac{r}{2} \\
            0 \leq j \leq \frac{s}{2} \\
            0 \leq k - i - j \leq \min\{r-2i, s-2j\}
        \end{array}
        \right\}.
    \]
\end{lemma}

\begin{proof}
    Routine, direct manipulations (see \cite[Prop.~1]{Stanley-LCC}) give
      \[ C(q) = A(q) B(q) = \sum_{i=0}^{\lfloor r/2\rfloor} \sum_{j=0}^{\lfloor s/2\rfloor} (a_i - a_{i-1}) (b_j - b_{j-1}) (q^i + \cdots + q^{r-i}) (q^j + \cdots + q^{s-j}). \]
    Observe
    \begin{align*}
      (1-q) (q^i + \cdots + q^{r-i}) (q^j + \cdots + q^{s-j})
        &= q^{i+j} (1 + \cdots + q^{\min\{r-2i, s-2j\}}) \\
        & -q^{r+s+1-i-j} (1 + \cdots + q^{-\min\{r-2i, s-2j\}}),
    \end{align*}
    where the positive terms occur at or before $(r+s)/2$ and the negative terms occur after $(r+s)/2$. Extracting the coefficient of $q^k$ from $(1-q) C(q)$ now gives the stated result.
\end{proof}

We may iterate \Cref{lem:unimodal-product} to get the following combinatorial interpretation of the successive differences of coefficients for products of $q$-integers.

\begin{lemma}\label{lem:q-int-product}
    Let $\mathbf{a} = (a_1, \ldots, a_t)$ be a vector of non-negative integers. Set
      \[ p(\mathbf{a}; q) \coloneqq \prod_{i=1}^t (1+q+\cdots+q^{a_i}) = \sum_{k=0}^{|\mathbf{a}|} c_k(\mathbf{a}) q^k \]
    where $|\mathbf{a}| \coloneqq a_1 + \cdots + a_t$. Then $p(\mathbf{a}; q)$ is symmetric and unimodal, with
    \begin{align*}
        &c_k(\mathbf{a}) - c_{k-1}(\mathbf{a}) \\
          &= \#\left\{(k_1, k_2, \ldots, k_t) \in \mathbb{Z}^t \;\middle|\;
          \begin{array}{l}
              0 = k_1 \leq k_2 \leq \cdots \leq k_t = k \\
              k_{i+1} - k_i \leq \min\{a_1+\cdots+a_i - 2k_i, a_{i+1}\} \text{ for } 1 \leq i \leq t-1
          \end{array}
          \right\}
    \end{align*}
    for $0 \leq k \leq |\mathbf{a}|/2$.
\end{lemma}

\begin{remark}
    In the case when all $a_i = 1$, we have $p(\mathbf{a}; q) = (1+q)^t$ and $c_k = \binom{t}{k}$. It is easy to see that these are unimodal, but it is a common experience to be surprised at the difficulty of finding a combinatorial interpretation of $\binom{t}{k} - \binom{t}{k-1}$. One standard approach involves interpreting the difference in terms of lattice paths using the reflection principle, which is already nontrivial. Another approach, which in principle goes back to Clebsch--Gordan, is given by Greene--Kleitman \cite{GK-Sperner,greene1978proof} and produces a symmetric chain decomposition for arbitrary $\mathbf{a}$, see~\cite{pak2019combinatorial} for a detailed discussion. In  modern terminology, this can be interpreted in terms of Kashiwara crystals \cite{BumpSchilling}. One takes a tensor product of linear crystals with $a_i+1$ elements for $i=1, \ldots, t$, so that $c_k(\mathbf{a})$ is the number of connected components with lowest weight $\leq k$ and hence $c_k(\mathbf{a}) - c_{k-1}(\mathbf{a})$ is the number of lowest weight elements of weight exactly $k$.
\end{remark}

\begin{proof}
    Induct on $t$. The base case $t=1$ is immediate. For $t > 1$, let $\overline{\mathbf{a}} \coloneqq (a_1, \ldots, a_{t-1})$. By \Cref{lem:unimodal-product},
    \begin{align*}
        c_k(\mathbf{a}) - c_{k-1}(\mathbf{a}) = \sum_{(i, 0) \in R_k(|\overline{\mathbf{a}}|, a_t)} (c_i(\overline{\mathbf{a}}) - c_{i-1}(\overline{\mathbf{a}})).
    \end{align*}
    Set $i = k_{t-1}$ and $k = k_t$, so that $(i, 0) \in R_k(|\overline{\mathbf{a}}|, a_t)$ becomes
      \[ 0 \leq k_t - k_{t-1} \leq \min\{|\overline{\mathbf{a}}| - 2k_{t-1}, a_t\}. \]
    The result now follows directly by induction.
\end{proof}

\subsection{Marked KOH trees and Kronecker coefficients}
We now turn to our combinatorial interpretation of \eqref{eq:stab-diff-2}. Recall the set of KOH trees $\mathcal{T}(n, k)$ from  \Cref{def:KOH-tree}. For any $T \in \mathcal{T}(n, k)$, we order the leaves $v_1, \ldots, v_t$ of $T$ by some fixed, arbitrary procedure, say depth first search from left to right as in \Cref{fig:KOH-tree}. The corresponding multiset of leaf labels is $\mathcal{L}(T) = \{a_i\}_{i=1}^t$. A \emph{marked KOH tree} is a KOH tree where additionally each leaf $v_i$ is marked with an integer $k_i$.

\begin{definition}
    Let $\mathcal{T}(n, k, r)$ denote the set of marked KOH trees of type $(n, k)$ where the leaves have labels $\{a_i\}$ and corresponding marks $\{k_i\}$ which satisfy
    \begin{align*}
        \#\left\{(k_1, k_2, \ldots, k_t) \in \mathbb{Z}^t \;\middle|\;
          \begin{array}{l}
              0 = k_1 \leq k_2 \leq \cdots \leq k_t = r - nk/2 + (a_1+\cdots+a_t)/2 \\
              k_{i+1} - k_i \leq \min\{a_1+\cdots+a_i - 2k_i, a_{i+1}\} \text{ for } 1 \leq i \leq t-1
          \end{array}
          \right\}.
    \end{align*}
\end{definition}

\Cref{thm:main_kron} is an immediate corollary of the following and \Cref{lem:stab-diff}.

\begin{theorem}[\Cref{thm:main_kron}]\label{thm:Kronecker}
    Let $r \leq nk/2$. Then
      \[ g((nk-r,r),(n^k),(n^k))= p_r(n, k) - p_{r-1}(n, k) = |\mathcal{T}(n, k, r)|. \]
\end{theorem}

\begin{proof}
    Combine \Cref{prop:KOH-trees} and \Cref{lem:q-int-product}. Here the quantity $r-nk/2+(a_1+\cdots+a_t)/2 = r-\sigma(T)/2$ comes from the $q$-shift in \eqref{eq:KOH-trees}.
\end{proof}

\begin{example}
    The KOH tree of type $(8, 9)$ in \Cref{fig:KOH-tree} (left) is redrawn in \Cref{fig:KOH-tree-marked} (left) with marks on the leaves. Here we require $0 \leq r \leq 36$ and $0=k_1 \leq k_2 \leq k_3=r-18$, indicating that this term of \Cref{prop:KOH-trees} does not contribute to the positive successive differences of $\binom{8+9}{9}_q$ until at least the $q^{18}$ coefficient. The remaining two conditions on the marks simplify to $k_2 \leq 0$ and $(r-18)-k_2 \leq 14-2k_2$, so $r \leq 32$. Thus there is precisely one marking for $18 \leq r \leq 32$, namely $(0, 0, r-18)$, and no markings otherwise.
    
    The second tree from \Cref{fig:KOH-tree} (right) is redrawn in \Cref{fig:KOH-tree-marked} (right). It has marks $0=k_1\leq k_2\leq k_3\leq k_4= r-70$, together with the additional constraints above. When $r=81$, these are equivalent to $k_2 \leq 2$, $k_2 + k_3 \leq 8$, and $k_3 \leq 9$, resulting in $21$ possible markings.
\end{example}

\begin{figure}[hbtp]
    \begin{center}
    \begin{tikzpicture}[font=\small]
        \node (root) {$(\,\vcenter{\hbox{\scalebox{0.4}{\ydiagram{4,3,1,1}}}},\,8,\,9)$};
        \node[below left=8mm and 6mm of root] (left) {$(\,\vcenter{\hbox{\scalebox{0.4}{\ydiagram{1,1}}}},\,2,\,2)$};
        \node[below=10mm of root] (middle) {$14$};
        \node[below=0mm of middle, circle, draw] (middlemark) {$k_2$};
        \node[below right=8mm and 6mm of root] (right) {$22$};
        \node[below=0mm of right, circle, draw] (rightmark) {$k_3$};
        \node[below=6mm of left] (leftchild) {$(\,\vcenter{\hbox{\scalebox{0.4}{\ydiagram{2}}}},\,0,\,2)$};
        \node[below=6mm of leftchild] (leftchildchild) {$0$};
        \node[below=0mm of leftchildchild, circle, draw] (leftchildchildmark) {$k_1$};
        
        \draw[thick, ->] (root) -- (left)        node[midway, xshift=-5pt, yshift=5pt] {$1$};
        \draw[thick, ->] (root) -- (middle)       node[midway, xshift=7pt] {$3$};
        \draw[thick, ->] (root) -- (right)       node[midway, xshift=5pt,  yshift=5pt] {$4$};
        \draw[thick, ->] (left) -- (leftchild)  node[midway, xshift=7pt] {$1$};
        \draw[thick, ->] (leftchild) -- (leftchildchild) node[midway, xshift=7pt] {$2$};
    \end{tikzpicture}
 $\qquad$   \begin{tikzpicture}[font=\small]
  \node (r) at (0,0.5) { $(\,\vcenter{\hbox{\scalebox{0.4}{\ydiagram{4,4,3,2,2,2}}}},\,12,\,17)$ };

  \node (n1) at (-2,-2) { $(\,\vcenter{\hbox{\scalebox{0.4}{\ydiagram{2,1}}}},\,4,\,3)$ };
  \node (n2) at (0,-2)  { $12$ };
          \node[below=0mm of n2, circle, draw] (n2mark) {$6$};

  \node (n3) at (2,-2)  { $(\,\vcenter{\hbox{\scalebox{0.4}{\ydiagram{2}}}},\,22,\,2)$ };

  \draw[thick, ->] (r) -- (n1) node[midway, above left] {$2$};
  \draw[thick, ->] (r) -- (n2) node[midway, right]      {$3$};
  \draw[thick, ->] (r) -- (n3) node[midway, above right]{$4$};

  \node (a1) at (-2.5,-3.5) {$2$};
  \node[below=0mm of a1, circle, draw] (a1mark) {$0$};
  \node (a2) at (-1,-3.5) {$6$};
  \node[below=0mm of a2, circle, draw] (a2mark) {$2$};
  \draw[thick, ->] (n1) -- (a1) node[midway, above left] {$1$};
  \draw[thick, ->] (n1) -- (a2) node[midway, above right]{$2$};

  \node (b2) at (2.5,-3.5) { $44$ };
  \node[below=0mm of b2, circle, draw] (b2mark) {$11$};
  \draw[thick, ->] (n3) -- (b2) node[midway, above right]{$2$};
\end{tikzpicture}
    \end{center}
    \caption{(\textsc{Left}) A marked KOH tree $T$ of type $(8, 9)$, with abstract marks circled. Leaves $((1), a, 1)$ are abbreviated as $a$. (\textsc{Right}) A marked KOH tree of type $(12,17)$, with valid marks $(k_1,k_2,k_3,k_4)=(0,2,6,11)$ for $r=81$.}
    \label{fig:KOH-tree-marked}
\end{figure}
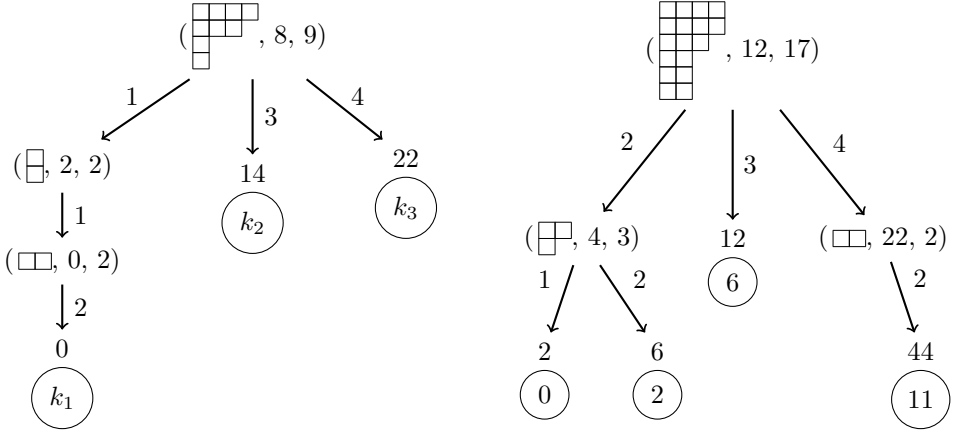

\section{The plethysm case}\label{sec:plethysm}

We now give our more general combinatorial interpretation of \eqref{eq:stab-diff-1}. Recall the set of GOS trees $\mathcal{G}(\lambda, k)$ from \Cref{def:GOS-tree}. As before, if $T \in \mathcal{G}(\lambda, k)$ has leaf vertices $v_1, \ldots, v_t$ and leaf multiset $\{a_i\}_{i=1}^t$ where leaves have been ordered by some fixed, arbitrary procedure, a \emph{marked GOS tree} additionally marks each leaf with an integer $k_i$. For concreteness, our procedure is depth-first search from left-to-right when ordered as in \Cref{fig:GOS-tree}.

\begin{definition}
    Let $\mathcal{G}(\lambda, k, r)$ denote the set of marked GOS trees of type $(\lambda, k)$ where the marks $\{k_i\}$ for the leaves labeled $\{a_i\}$ satisfy
    \begin{align*}
        \#\left\{(k_1, k_2, \ldots, k_t) \in \mathbb{Z}^t \;\middle|\;
          \begin{array}{l}
              0 = k_1 \leq k_2 \leq \cdots \leq k_t = r - |\lambda|k/2 + (a_1+\cdots+a_t)/2 \\
              k_{i+1} - k_i \leq \min\{a_1+\cdots+a_i - 2k_i, a_{i+1}\} \text{ for } 1 \leq i \leq t-1
          \end{array}
          \right\}.
    \end{align*}
\end{definition}

\Cref{thm:main_pleth} is an immediate corollary of the following and \Cref{lem:stab-diff}.

\begin{theorem}[\Cref{thm:main_pleth}]\label{thm:plethysm}
    Let $\mu$ be a partition and $k \geq 1$. Suppose $r \leq |\mu|k/2$. Then
      \[ [q^r] s_\mu(1, q, \ldots, q^k) - [q^{r-1}] s_\mu(1, q, \ldots, q^k) = |\mathcal{G}(\mu, k, r)|. \]
\end{theorem}

\begin{proof}
    Combine \Cref{prop:GOS-trees} and \Cref{lem:q-int-product}.
\end{proof}

\begin{figure}[H]
    \begin{center}
    \begin{tikzpicture}
        \node (root) {$(\,\vcenter{\hbox{\scalebox{0.4}{\ydiagram{3,3,2,1}}}},\,((1^9), (2, 1^4), (1^3), (1), \varnothing)),\,6)$};
        \node[below left=8mm and 20mm of root] (a) {$(\,\vcenter{\hbox{\scalebox{0.4}{\ydiagram{3,1}}}},\,2,\,4)$};
        \node[below left=8mm and -10mm of root] (b) {$0$};
        \node[below=8mm of root] (c) {$(\,\vcenter{\hbox{\scalebox{0.4}{\ydiagram{3}}}},\,0,\,3)$};
        \node[below right=8mm and -10mm of root] (d) {$1$};
        \node[below right=8mm and 20mm of root] (e) {$9$};

        \node[below left=8mm and 0mm of a] (aa) {$0$};
        \node[below right=8mm and 0mm of a] (ab) {$4$};
        \node[below=8mm of c] (ca) {$0$};

        \node[below=0mm of aa, circle, draw] (middlemark) {$k_1$};
        \node[below=0mm of ab, circle, draw] (middlemark) {$k_2$};
        \node[below=0mm of b, circle, draw] (middlemark) {$k_3$};
        \node[below=0mm of ca, circle, draw] (middlemark) {$k_4$};
        \node[below=0mm of d, circle, draw] (middlemark) {$k_5$};
        \node[below=0mm of e, circle, draw] (middlemark) {$k_6$};

        \draw[thick, ->] (root) -- (a)        node[midway, xshift=-5pt, yshift=5pt] {$1,1$};
        \draw[thick, ->] (root) -- (b)        node[midway, xshift=-7pt, yshift=5pt] {$1,2$};
        \draw[thick, ->] (root) -- (c)        node[midway, xshift=-8pt, yshift=5pt] {$2,1$};
        \draw[thick, ->] (root) -- (d)        node[midway, xshift=-20pt, yshift=5pt] {$3,1$};
        \draw[thick, ->] (root) -- (e)        node[midway, xshift= 5pt, yshift=5pt] {};

        \draw[thick, ->] (a) -- (aa)          node[midway, xshift=-5pt, yshift=5pt] {$1$};
        \draw[thick, ->] (a) -- (ab)          node[midway, xshift= 5pt, yshift=5pt] {$3$};
        \draw[thick, ->] (c) -- (ca)          node[midway, xshift=-5pt, yshift=5pt] {$3$};
    \end{tikzpicture}
    \end{center}
    \caption{A marked GOS tree $T$ of type $((3, 3, 2, 1), 6)$. Leaves $((1), a, 1)$ are abbreviated as $a$. Marks are circled.}
    \label{fig:GOS-tree-marked}
\end{figure}
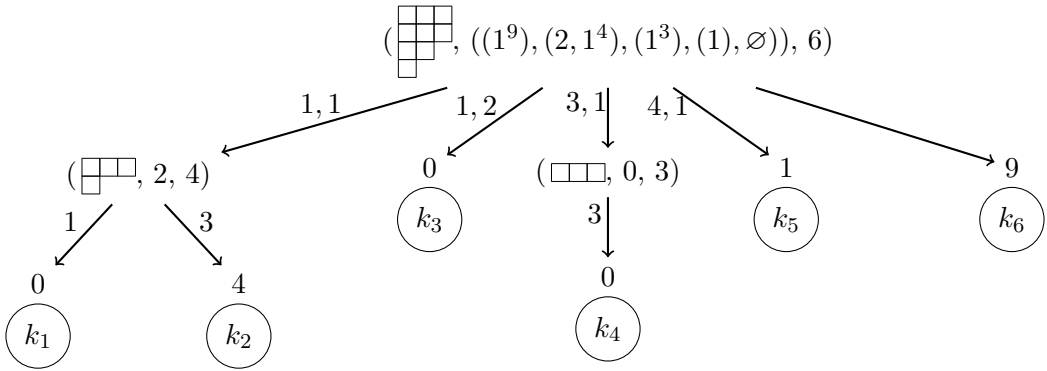

\begin{example}\label{ex:GOS}
    The GOS tree of type $((3, 3, 2, 1), 6)$ in \Cref{fig:GOS-tree} is redrawn in \Cref{fig:GOS-tree-marked} with marks on the leaves. Here we require $0 \leq r \leq 27$ and $0=k_1 \leq k_2 \leq k_3 \leq k_4 \leq k_5 \leq k_6=r-27+7=r-20$, indicating that this term of \Cref{prop:GOS-trees} does not contribute to the positive successive differences of $s_{(3, 3, 2, 1)}(1, q, \ldots, q^6)$ until at least the $q^{20}$ coefficient. The remaining conditions on marks further restrict the set of $r$ for which this tree is in $\mathcal{G}((3, 3, 2, 1), 6, r)$. Over all relevant marks, the tree contributes the following:
    \begin{align*}
        q^{20} [5]_q [2]_q [10]_q
          &= \quad\qquad {\scriptstyle q^{20}+3q^{21}+5q^{22}+7q^{23}+9q^{24}+10q^{25}+10q^{26}+10q^{27}} \\
          &{\scriptstyle +10q^{28}+10q^{29}+9q^{30}+7q^{31}+5q^{32}+3q^{33}+q^{34}} \\
        s_{(3, 3, 2, 1)}(1, q, \ldots, q^6)
          &=\ \scriptstyle{q^{10}+3q^{11}+7q^{12}+15q^{13}+28q^{14}+48q^{15}+78q^{16}+118q^{17}+169q^{18}} \\
          &\scriptstyle{+232q^{19}+304q^{20}+382q^{21}+463q^{22}+540q^{23}+607q^{24}+661q^{25}+695q^{26}+706q^{27}} \\
          &\scriptstyle{+695q^{28}+661q^{29}+607q^{30}+540q^{31}+463q^{32}+382q^{33}+304q^{34}+232q^{35}+169q^{36}} \\
          &\scriptstyle{+118q^{37}+78q^{38}+48q^{39}+28q^{40}+15q^{41}+7q^{42}+3q^{43}+q^{44}.}
    \end{align*}
\end{example}

\subsection*{Acknowledgements}

The authors would like to thank  \'Alvaro Guti\'errez, Christian Ikenmeyer, Rosa Orellana, Anne Schilling, and Mike Zabrocki for fruitful conversations, as well as the anonymous referees for helpful comments. Pak was partially supported by NSF grant CCF:2302173, Panova was partially supported by NSF grant CCF:2302174 and an AMS Birman fellowship. Swanson was partially supported by NSF grant DMS-2348843. This work was completed while Panova and Swanson were in residence at ICERM, which is supported by NSF grant DMS-1929284, during the Fall 2025 Semester Program on Categorification and Computation in Algebraic Combinatorics.

\bibliographystyle{amsalphavar}
\bibliography{main}

\end{document}